\documentclass[a4paper,12pt]{article}
\usepackage[utf8]{inputenc}
\usepackage[english]{babel}
\usepackage{enumitem}
\usepackage{color}
\usepackage{amsmath}
\usepackage{amssymb}
\usepackage{amsthm}
\usepackage{amsfonts}
\usepackage{epsfig}
\usepackage{epstopdf}
\usepackage{graphicx}
\usepackage{multirow}

\usepackage{caption}
\newtheorem{theorem}{Theorem}[section]
\newtheorem{definition}{Definition}[section]
\newtheorem{lemma}{Lemma}[section]

\newtheorem{remark}{Remark}[section]

\newtheorem{algorithm}{Algorithm}[section]

\theoremstyle{definition}

\DeclareMathOperator*{\argmin}{arg\,min}

\DeclareMathOperator*{\Argmin}{Arg\,min}
\DeclareMathOperator*{\Max}{Max}

\allowdisplaybreaks

\begin{document}
\title{\textbf{Contractive difference-of-convex algorithms}}
\author{ \sc Songnian He$^{a}${\thanks{ email: songnianhe@163.com}},\,\,
Qiao-Li Dong$^{a}${\thanks{Corresponding author. email: dongql@lsec.cc.ac.cn}}\,\,\,and\,
Michael Th. Rassias$^{b,c}${\thanks{email: mthrassias@yahoo.com}}\\
\small $^{a}$Tianjin Key Laboratory for Advanced Signal Processing and College of Science, \\
\small Civil Aviation University of China, Tianjin 300300, China,\\
\small $^{b}$Department of Mathematics and Engineering Sciences,
Hellenic Military Academy, \\
\small 16673 Vari Attikis, Greece\\
\small $^{c}$Institute for Advanced Study, Program in Interdisciplinary Studies, 1 Einstein Dr, \\
\small Princeton, NJ 08540, USA.
}

\date{}

\maketitle

\begin{abstract}

{The difference-of-convex algorithm (DCA) and its variants are the most popular methods to solve the difference-of-convex  optimization problem. Each iteration of them is reduced to  a convex optimization problem, which generally needs  to be solved by iterative methods such as proximal gradient algorithm. However, these  algorithms essentially belong to some iterative methods of fixed point problems of averaged  mappings, and their convergence speed is generally slow. Furthermore, there is seldom research on the  termination rule of these iterative algorithms solving the subproblem of DCA. To overcome these defects, we firstly show that the subproblem of the linearized proximal method (LPM) in each iteration is equal to the fixed point problem of a contraction. Secondly, by using Picard iteration  to approximately solve the subproblem of LPM in each iteration, we propose a contractive difference-of-convex  algorithm (cDCA) where    an adaptive termination rule is presented.} Both global subsequential convergence and  global convergence of the whole sequence of cDCA are established.  Finally, preliminary results from numerical experiments
are promising.
\end{abstract}

\vskip 2mm


\date{}

\section{ Introduction }

The difference-of-convex (DC) optimization problem   refers to {minimizing} the difference of two convex
functions, which forms a large class of nonconvex optimization and global optimization.  It models real world  optimization problems
 \cite{AT}, such as digital communication system \cite{Alvarado}, compressed sensing \cite{Yin-Lou-He-Xin}, machine learning, data mining \cite{LeThi},
location, distance geometry and clustering \cite{Tuy}, and hyperparameter selection \cite{Ye}. Interested
readers may consult the review articles \cite{dO,Pham-Le2,{Le-Pham}}.

In this paper we focus on a class of  DC optimization problem as follows:
\begin{equation}
\label{DC}
\min_{x\in \mathbb{R}^n }F(x):=f(x)+g(x)-h(x),
\end{equation}
{where  $f: \mathbb{R}^n \rightarrow \mathbb{R}$ is a smooth convex function with a Lipschitz continuous gradient whose Lipschitz continuity constant is $L_f > 0,$
 $g: \mathbb{R}^n\rightarrow (-\infty,+\infty]$ is a  proper,  convex, lower semi-continuous  and prox friendly function and
 $h: \mathbb{R}^n\rightarrow \mathbb{R}$ is a convex function.
We assume that $F$ is level-bounded, i.e., for each $r\in \mathbb{R}$, the set $\{x \in \mathbb{R}^n:
F(x)\leq r\}$ is bounded.
Denote by $\mathcal{S}$  the solution set of the problem (\ref{DC}). It is easy to see that $\mathcal{S}\neq \emptyset$ under the above assumption on the objective.}
The problem \eqref{DC} has applications in sparse learning and compressed
sensing, where  $f$ is a loss function representing data fidelity, while $g-h$
is a regularizer for inducing desirable structures (for example, sparsity) in the solution (see, e.g. \cite{Liu-Pong-Takeda,{Lou-Yan}} for details).

{It is well-known that the difference-of-convex algorithm (DCA)  and its various variants were designed to solve the problem \eqref{DC}.} Given the current iterate $x^k,$ DCA generates the next one by solving the convex optimization problem
\begin{equation}
\label{DCA}
x^{k+1}{\in\Argmin_{x\in \mathbb{R}^n}}\left\{f(x)+g(x)-\langle\eta^k, x\rangle\right\},\\
\end{equation}
where $\eta^k\in \partial h(x^k)$ is taken arbitrarily.
It has received a great deal of attention of
researchers, who improved it in various ways; see, e.g., \cite{AV,Pham-Souad,{Pham-Le},{Pham-Le1},{Banert-Bot},{Liu-Pong-Takeda1}}  and references therein.
By exploring the structure of the sum of $f$ and $g$, the proximal DCA (pDCA) was recently introduced for the problem \eqref{DC} (e.g., see \cite{Gotoh}). Given the current iterate $x^k,$ the new iterate is obtained by solving the proximal subproblem
\begin{equation}
\label{pDCA}
x^{k+1}=\argmin_{x\in \mathbb{R}^n}\left\{\langle \nabla f(x^k)-\eta^k,x\rangle
+\frac{L_f}2\|x-x^k\|^2+g(x)\right\}
\end{equation}
where $\eta^k\in \partial h(x^k)$ is taken arbitrarily.
To speed up the pDCA by  Nesterov's extrapolation technique \cite{Nesterov}, Wen et al. \cite{Wen} recently proposed a proximal DCA with extrapolation (pDCA$_{e}$)
\begin{equation}
\label{pDCAe}
\left\{
\aligned
&y^k=x^k+\beta_k(x^k-x^{k-1}),\\
&x^{k+1}=\argmin_{x\in \mathbb{R}^n}\left\{\langle \nabla f(y^k)-\eta^k,x\rangle
+\frac{L_f}2\|x-y^k\|^2+g(x)\right\},\\
\endaligned
\right.
\end{equation}
where $\eta^k\in \partial h(x^k)$ is taken arbitrarily and the extrapolation parameters $\{\beta_k\}_{k=0}^\infty\subseteq[0,1)$ satisfy $\sup_k\beta<1.$ Meanwhile,
Phan et al. \cite{PhanLe} introduced an accelerated DCA (ADCA)
\begin{equation}
\label{ADCA}
\left\{
\aligned
&\hbox{Choose}\,x^0\in \mathbb{R}^n,\,z^0=x^0,\,q\in \mathbb{N},\,\rho>L_f,\\
&z^k=x^k+\beta_k(x^k-x^{k-1}),\,\,\hbox{if}\,\,k\ge1,\\
&v^k=
\begin{cases}
z^k,&\hbox{if}\,\, F(z^k)\le\max_{t=\max{(0,k-q)},...,k}F(x^t),\\
x^k,&\hbox{otherwise},\\
\end{cases}\\
&y^k=\rho v^k-\nabla f(v^k)+\eta^k,\,\,\hbox{where}\,\,\eta^k\in \partial h(v^k),\\
&x^{k+1}=\argmin_{x\in  \mathbb{R}^n}\frac\rho2\|x\|^2+g(x)-\langle y^k,x\rangle,
\endaligned
\right.
\end{equation}
where the extrapolation parameters $\{\beta_k\}_{k=0}^\infty$  are taken as in FISTA \cite{BT}.
The choice of $v^k$ allows the objective function $F$ to increase and consequently to escape from a
potential bad local minimum. Further progress on {pDCA} can be founded in \cite{Liu-Pong-Takeda,LL,{Lu-Zhou},{Lu-Zhou-Sun},Pang-Razaviyayn-Alvarado}.

Recently, some  methods were introduced to solve more general   DC \eqref{DC}. 
Yu et al. \cite{Yu-Pong-Lu2021} proposed a sequential convex programming method with line search, which doesn't need  $f$  to be convex.
A backward-Douglas--Rachford method was presented in \cite{Pham} to solve DC problem where $f$ and $g$ {do not} require  to be convex. The functions $g$ in \cite{Yu-Pong-Lu2021} and $f$, $g$ and $h$ in \cite{Pham}  are assumed to be prox friendly. Syrtseva et al. \cite{Syrtseva} introduced a proximal bundle method which only needs $[h(x)-f(x)]$ to be weakly convex.
It solves a strongly convex QP per iteration and {does not} need the proximal operators of the functions.

{One common feature of DCA and its variants is that each iteration is reduced to a convex optimization problem  generally solving by iterative algorithms.  However, these  algorithms essentially belong to some iterative methods of fixed points of averaged  mappings, and consequently their convergence speed is generally slow.}
{For example, the subproblem \eqref{DCA} is convex and generally does not have a closed-form solution. Therefore, one needs to solve it by using iterative algorithms, such as  proximal gradient methods. In fact, by the first-order optimality condition, it is easy to write \eqref{DCA} as follows:
\begin{equation}
\label{averaged}
x^{k+1}={\rm Prox}_{\frac1{L_f}g}\big[x^{k+1}-\frac1{L_f}\nabla f(x^{k+1})+\frac1{L_f}\eta^k\big].
\end{equation}
The operator corresponding to the implicit scheme \eqref{averaged} is ${\rm Prox}_{\frac1{L_f}g}\big[(I-\frac1{L_f}\nabla f)\cdot+\frac1{L_f}\eta^k)\big]$ which is $\frac34$-averaged (see \cite{XuHK}). Therefore $x^{k+1}$ in  \eqref{averaged} can be approximated  by Krasnosel'ski\v{\i}--Mann iteration whose convergence may be arbitrarily slow (see, e.g., \cite{Oblomskaja}). As far as we know, there are no results on iterative termination rule for Krasnosel'ski\v{\i}--Mann iteration approximating \eqref{averaged}.
\vskip 1mm

By imposing a regularizer in the subproblem of DCA \eqref{DCA}, Sun et al. \cite{Sun} presented 
the linearized proximal method (LPM) as follows:
\begin{equation}\label{eq:MDC}
x^{k+1}=\argmin_{x\in \mathbb{R}^n}\left\{f(x)+g(x)-\langle\eta^k, x\rangle+\frac{\lambda}{2}\|x-x^k\|^2\right\},
\end{equation}
where $\eta^k\in \partial {h}(x^k)$ and $\lambda$ is an arbitrary positive number. However, it is usually difficult to obtain the analytical solution of the subproblem of LPM in each iteration.
In this paper, we firstly show that the subproblem of LPM is  equal to the fixed point problem of a  contraction. Secondly, based on this result, we introduce a  contractive difference-of-convex  algorithm (cDCA), whose core is to    approximately solve the subproblem of LPM in each iteration by Picard iteration. An adaptive iterative termination rule is given. The global subsequential convergence and  global convergence of the whole sequence of cDCA  are shown under appropriate conditions.
\vskip 2mm

The paper is organized as follows. Section 2 reviews some preliminary
notions, definitions and lemmas which will be helpful for further analysis. In section 3, we present the
motivation of cDCA and  introduce cDCA.  Section 4   proves the global subsequential convergence and  global convergence of the whole sequence of cDCA.  Section 5 contains
preliminary results of numerical experiments. Finally, in Section 6, we give some conclusive remarks.
}
\section{ Preliminaries}

Throughout this paper, we use $\mathbb{R}^n$ to denote the $n$-dimensional Euclidean space with inner product $\langle \cdot, \cdot\rangle$ and Euclidean norm $\|\cdot\|$.
Denote by $I$ the identity operator from $\mathbb{R}^n$ into itself. We use $x^k\rightarrow x$ to indicate that the sequence $\{x^k\}_{k=0}^\infty$ converges to  $x$,  $\omega(x^k) =\{x\mid\exists \, \{x^{k_l}\}_{l=0}^\infty\subseteq\{x^k\}_{k=0}^\infty$ such that $ x^{k_l} \rightarrow x\}$ to denote the  $\omega$-limit set of $ \{x^k\}_{k=0}^\infty$ {and  ${\rm co}(\{x^k\}_{k=0}^\infty)$ to denote the convex hull of  $\{x^k\}_{k=0}^\infty$.}
\vskip 1mm

Now we recall some definitions of the mappings.   A mapping $T:\mathbb{R}^n\rightarrow \mathbb{R}^n$ is called $\beta$-cocoercive (also called $\beta$-inverse strongly monotone), if there exists a constant $ \beta > 0$, such that
$$
\langle T(x)-T(y),x-y \rangle \geq \beta \|T(x)-T(y)\|^2,\quad\forall x, y\in \mathbb{R}^n.
$$
{A mapping $T:\mathbb{R}^n\rightarrow \mathbb{R}^n$ is called firmly nonexpansive if $\beta=1.$}
A mapping $T:\mathbb{R}^n\rightarrow \mathbb{R}^n$ is called $L $-Lipschitz continuous, if there exists a constant $ L > 0$, such that
$$
\|T(x)-T(y)\|\leq  L\|x-y\|,\quad\forall x, y\in \mathbb{R}^n.
$$
Then {$T$ is called a $L $-contraction  if $L\in[0,1)$} and  nonexpansive  if $L=1$.
A mapping $T:\mathbb{R}^n\rightarrow \mathbb{R}^n$ is called $\alpha$-averaged ($0<\alpha<1$) if there exists a nonexpansive operator
$S$ such
that $T=(1-\alpha)I+\alpha S.$

\begin{lemma}\label{lem22.2} {\rm{(\cite[Proposition 3.4]{XuHK}})}
Let $T: \mathbb{R}^n\rightarrow \mathbb{R}^n$ be  $\beta$-cocoercive. Then for any $\nu\in (0, 2\beta]$, $I-\nu T$ is a nonexpansive mapping.
\end{lemma}

Let $\varphi:\mathbb{R}^n\rightarrow (-\infty,+\infty]$ be a function.
 The domain of $\varphi$
is ${\rm dom}(\varphi)=\{x\in \mathbb{R}^n\,|\,\varphi(x)<+\infty\}$.
A function $\varphi$ is
called  proper if  ${\rm dom}(\varphi)\neq\emptyset$.
A function $\varphi$ is called {convex} if
$$
\varphi(\lambda u+(1-\lambda)v)\leq\lambda \varphi(u)+(1-\lambda)\varphi(v),~\forall\lambda\in[0,1],~\forall u,v\in \mathbb{R}^n.
$$
A function $\varphi:\mathbb{R}^n\rightarrow (-\infty,+\infty]$ is called  {lower semi-continuous}
($lsc$)  if for each $u\in \mathbb{R}^n$, $ u^{k}\rightarrow u$ implies
$$
 \varphi(u)\leq\liminf_{k\rightarrow\infty}\varphi(u^{k}).
$$

Let $\varphi:\mathbb{R}^n\rightarrow (-\infty,+\infty]$ be proper. The {subdifferential} of $\varphi$ at $u\in\mathbb{R}^n$ is
\begin{equation}\label{SI}
\partial\varphi(u)=\{\xi\in \mathbb{R}^n\,|\,\varphi(z)\geq \varphi(u)+\langle \xi,z-u\rangle,~~\forall z\in \mathbb{R}^n\}.
\end{equation}
Then $\varphi$ is called {subdifferentiable} at $u$ if $\partial\varphi(u)\neq\emptyset$. The inequality in \eqref{SI} is called the {subdifferential inequality} of $\varphi$ at $u$. In addition, if $\varphi$ is continuously differentiable,
then the subdifferential \eqref{SI} reduces to the gradient of $\varphi$, denoted by $\nabla\varphi$.

\begin{lemma}{\rm(\cite[Proposition 16.20]{BC2011})}\label{le2.22}
Every convex function defined on $\mathbb{R}^n$ with finite values is continuous and subdifferentiable, and its subdifferential operator is  bounded, i.e.,  bounded on bounded sets.
\end{lemma}

\begin{lemma}{\rm(\cite[Theorem 24.4]{Rock})}
\label{le2.3}
Let  $\varphi:~\mathbb{R}^n\rightarrow (-\infty,+\infty]$ be a proper, lower semi-continuous and convex function. If $\{x^k\}_{k=0}^\infty$ and $\{\xi^k\}_{k=0}^\infty$ are sequences such that  $\xi^k\in \partial \varphi (x^k)$, where   $x^k\rightarrow x$ and $\xi^k\rightarrow \xi$ as $k\rightarrow \infty$, then $\xi\in \partial \varphi (x).$
\end{lemma}

%\vskip 2mm
%The following lemma provides the boundedness and continuity properties of the subdifferential,
%see, e.g., \cite{Bertsekas}:
%
%\begin{lemma}
%Let  $\varphi:\mathbb{R}^n\rightarrow (-\infty,+\infty)$ be a convex function.
%If the sequence $\{x^k\}_{k=0}^\infty$ converges to $x\in \mathbb{R}^n$ and
%$\xi^k\in \partial \varphi (x^k)$, then the sequence   $\{\xi^k\}_{k=0}^\infty$
%is bounded and each of its limit points is a subgradient of $\varphi$ at $x.$
%\end{lemma}

Let $\varphi: \mathbb{R}^n\rightarrow (-\infty,+\infty]$ be proper, lower semi-continuous, and convex and $\alpha > 0$ be an arbitrary positive constant. Then the proximal mapping (or proximity operator) of $\varphi$ parameterised by $\alpha$, denoted by ${\rm Prox}_{\alpha\varphi}$, maps every $x\in\mathbb{R}^n$ to the unique  minimum point of $\varphi+\frac{1}{2\alpha}\|x-\cdot\|^2;$ i.e.,
$$
{\rm Prox}_{\alpha\varphi}(x)=\argmin_{y\in\mathbb{R}^n} [\varphi(y)+\frac{1}{2\alpha}\|x-y\|^2],\quad\forall x\in \mathbb{R}^n.
$$

\begin{lemma}{\rm(\cite[Example 23.3 and 12.28]{BC2011})} 
\label{le2.6}
Let $\varphi: \mathbb{R}^n\rightarrow (-\infty,+\infty]$ be proper, lower semi-continuous, and convex and let $\alpha$ be an arbitrary positive constant. Then the following conclusions hold:
\vskip 1mm

\noindent
{\rm (i)} for $x, p\in \mathbb{R}^n,$  $p={\rm Prox}_{\alpha\varphi}(x)$ if and only if $\frac{x-p}{\alpha}\in \partial \varphi(p)$;
\vskip 1mm

\noindent
{\rm (ii)} ${\rm Prox}_{\alpha\varphi}$ is firmly nonexpansive.
\end{lemma}

\begin{lemma}{\rm(\cite[Corollary 18.17]{BC2011})}\label{lem22.4}
If $\varphi: \mathbb{R}^n \rightarrow \mathbb{R}$ is a convex function with $L$-Lipschitz continuous gradient $\nabla \varphi$, then its gradient $\nabla \varphi$ is $\frac{1}{L}$-cocoercive.
\end{lemma}

Let $\varphi: \mathbb{R}^n\rightarrow (-\infty,+\infty]$ be a proper lower semi-continuous function. For $-\infty < \rho_1<\rho_2\leq +\infty$,  set
$$
[\rho_1<\varphi<\rho_2]=\{x\in \mathbb{R}^n \,|\, \rho_1<\varphi(x)<\rho_2\}.
$$
\begin{definition}
\label{def21}
{\rm
The function $\varphi$ is said to have the Kurdyka--Łojasiewicz property at
$\hat{x}\in {\rm dom}(\partial \varphi) $ if there exist $\rho\in (0, +\infty],$
a neighborhood $U$ of $\hat{x}$ and a continuous concave function $\phi: [0, \rho)
\rightarrow \mathbb{R}_+$, such that:\vskip 1mm

\noindent
{\rm (i)} $\phi(0)=0$,\vskip 1mm

\noindent
{\rm (ii)} $\phi$ is $C^1$ on $(0, \rho)$,\vskip 1mm

\noindent
{\rm (iii)} for all $s\in (0, \rho)$, $\phi^\prime(s)>0$,\vskip 1mm

\noindent
{\rm (iv)} for all $x\in U\cap [\varphi(\hat{x})<\varphi<\varphi(\hat{x})+\rho]$,
the Kurdyka--Łojasiewicz inequality holds:
\begin{equation}\label{KL}
\phi^\prime(\varphi(x)-\varphi(\hat{x})){\rm dist}(0, \partial \varphi(x))\geq 1.
\end{equation}
}
\end{definition}

\section{The  contractive difference-of-convex algorithm}
{In this section, we will explain our idea and motivation first. Then, we propose the
contractive difference-of-convex algorithm (cDCA).

\begin{definition}
{\rm
We say that $\hat{x}\in \mathbb{R}^n$ is a critical point of $F:=f+g-h$ if
$$
{{0}}\in \nabla f(\hat{x})+\partial g(\hat{x})-\partial h(\hat{x}).
$$
The set of all critical points of $F$ is denoted by $\mathcal{X}$.
}
\end{definition}

\subsection{Motivation}

It is obvious that the subproblem  \eqref{DCA} may have more than one solution. To obtain a unique solution, the linearized proximal
 method was introduced for solving the problem \eqref{DC} as follows (see, e.g., \cite{Sun,Pang-Razaviyayn-Alvarado}).
\begin{algorithm}{\rm(LPM: The linearized proximal method)
\label{al:3.1}
\quad

\hspace*{0.1 pc} Step 0: Take  {$x^{0}\in \mathbb{R}^n$} and select $\lambda >0$ arbitrarily. Set $k:=0$.\\
\hspace*{1.5 pc} Step 1: For the iterate $x^k,$  take any $\eta^k\in \partial {h}(x^k)$, and compute $x^{k+1}$ via
\begin{equation}\label{eq:4.1}
x^{k+1}=\argmin_{x\in \mathbb{R}^n}\{f(x)+g(x)-\langle\eta^k, x\rangle+\frac{\lambda}{2}\|x-x^k\|^2\},
\end{equation}
\hspace*{1.5 pc} Step 2:  Set $k:=k+1$ and return to Step 1.
}
\end{algorithm}

Note that the subproblem   (\ref{eq:4.1}) has a unique solution since its objective is strongly convex.
The convergence of Algorithm 3.1 is easy to prove by the conventional analysis, see, \cite[Theorem 3.2]{Sun}.
\vskip 2mm

The following lemma shows that the subproblem (\ref{eq:4.1}) is equivalent to the fixed point problem of a  contraction.

\begin{lemma}
\label{lem31}
{For $x^k,\eta^k$ in \eqref{eq:4.1}, let $\lambda,\,\mu>0$ and define $T_k^{\lambda,\mu}: \mathbb{R}^n\rightarrow \mathbb{R}^n$ by
\begin{equation}\label{eq:3.b}
T_k^{\lambda,\mu}(x)=(1-\mu\lambda){x}-\mu\nabla f({x})+\mu\lambda x^k+\mu\eta^k,\quad\forall x\in \mathbb{R}^n.
\end{equation}
Then the following conclusions are valid.}
\begin{itemize}
\item[{\rm(i)}] The subproblem \eqref{eq:4.1} is equivalent to the fixed point problem of finding ${x}^{k+1}$ such that
\begin{equation}\label{contraction}
{\rm Prox}_{\mu g}[T_k^{\lambda,\mu}]x^{k+1}={x}^{k+1}.
\end{equation}

\item[{\rm(ii)}] For any $\lambda>0$ and $\mu\in (0, \frac{2}{2\lambda+L_f}]$,
$
{\rm Prox}_{\mu g}[T_k^{\lambda,\mu}]: \mathbb{R}^n\rightarrow \mathbb{R}^n
$
 is a  contraction with coefficient $1-\mu\lambda.$ For given $\lambda>0$, when $\mu=\frac{2}{2\lambda+L_f},$ the  coefficient reaches its minimum $\frac{L_f}{2\lambda+L_f}$.
\end{itemize}
 \end{lemma}

\begin{proof}
(i) For arbitrary $\lambda>0$ and $\mu>0$, we deduce that
\begin{equation}\label{eq:3.c}
\aligned
{x}^{k+1}\,&{\rm solves\,the\, subproblem}\, (\ref{eq:4.1})\\
&\Longleftrightarrow {{0}}\in \nabla f({x}^{k+1})+\lambda ({x}^{k+1}-x^k)+\partial g({x}^{k+1})-\eta^k\\
& \Longleftrightarrow -\nabla f({x}^{k+1})+\lambda (x^k-{x}^{k+1})+\eta^k\in \partial g({x}^{k+1}) \\
& \Longleftrightarrow {x}^{k+1}-\mu\nabla f({x}^{k+1})+\mu \lambda(x^k-{x}^{k+1})+\mu\eta^k\in {x}^{k+1}+\mu\partial g({x}^{k+1})\\
& \Longleftrightarrow {\rm Prox}_{\mu g}[(1-\mu\lambda){x}^{k+1}-\mu\nabla f({x}^{k+1})+\mu\lambda x^k+\mu\eta^k]={x}^{k+1}.
 \endaligned
 \end{equation}

(ii) For any $\lambda>0$ and $\mu\in (0, \frac{2}{2\lambda+L_f}]$, it is easy to verify that  $\frac{\mu}{1-\mu\lambda}\leq\frac{2}{L_f}$ holds. By using Lemma \ref{lem22.4}, Lemma \ref{lem22.2} and Lemma \ref{le2.6}, we assert that ${\rm Prox}_{\mu g}[T_k^{\lambda,\mu}]: {\mathbb{R}^n\rightarrow \mathbb{R}^n}:x\mapsto{\rm Prox}_{\mu g}[(1-\mu\lambda)(I-\frac{\mu}{1-\mu\lambda}\nabla f)x+\mu\lambda x^k+\mu\eta^k]$ is a  contraction with the coefficient $1-\mu\lambda$.
Obviously, for given $\lambda>0$,  the  coefficient reaches its minimum $\frac{L_f}{2\lambda+L_f}$ when $\mu=\frac{2}{2\lambda+L_f}$.
This completes the proof.
\end{proof}

{In view of Lemma \ref{lem31}, the  solution ${x}^{k+1}$ of the subproblem (\ref{eq:4.1}) is the  fixed point of the  contraction ${\rm Prox}_{\mu g}[T_{k}^{\lambda,\mu}]$ with $\mu\in (0, \frac{2}{2\lambda+L_f}]$. Therefore, we can use the following scheme to solve the subproblem \eqref{eq:4.1}
\begin{equation}
\label{eqd1}
{x}^{k+1}={\rm Prox}_{\mu g}\left[(1-\mu\lambda){x}^{k+1}-\mu\nabla f({x}^{k+1})+\mu\lambda x^k+\mu\eta^k\right].
\end{equation}
Note that the scheme \eqref{eqd1} is implicit in the sense that $x^{k+1}$
appears in both sides of \eqref{eqd1}. Thus, in general, $x^{k+1}$ can not be solved
exactly via \eqref{eqd1}. By Lemma \ref{lem31} and Banach contraction mapping principle, the solution of the subproblem (\ref{eq:4.1}) can be approximated  by Picard iteration at any accuracy.
However, high accuracy generally means a large number of iterations and the prohibitive computational cost.
A popular way to overcome this difficulty is to introduce an iteration termination rule so that the iteration will stop within finite steps. One may ask the question: how do we design  an iteration termination rule to guarantee that the sequence $\{x^{k}\}_{k=0}^\infty$  converges to a solution of the problem \eqref{DC}?

\subsection{Algorithm}

Now, by using Picard iteration to approximate $x^{k+1}$ in \eqref{eqd1}, we introduce a contractive difference-of-convex algorithm where    an adaptive termination rule is presented.

\begin{algorithm}{\rm(cDCA: The contractive difference-of-convex algorithm)
\label{al4.2}
%\vskip -2mm

\hspace*{0.1 pc} {Step 0: Take {${\rm tol}>0$}, $x^{-1},x^{0}\in \mathbb{R}^n$ with $x^{-1}\neq x^0$.  Select  $\lambda > 0$ and $\delta\in (0, \frac{2\lambda}{L_f})$
 \\
 \hspace*{5.0 pc}  arbitrarily, and set $\mu=\frac{2}{2\lambda+L_f}$. Set $k:=0$.\\}
\hspace*{1.5 pc} Step 1:  For the iterate $x^k,$  select $x_0^{k+1}$,  take any  $\eta^k\in \partial h(x^k)$, and   calculate\\
\begin{equation}\label{eq:4.102a}
x_{m}^{k+1}={\rm Prox}_{\mu g}\left[(1-\mu\lambda)x_{m-1}^{k+1}-\mu\nabla f(x_{m-1}^{k+1})+\mu\lambda x^k+\mu\eta^k\right],\,\,m=1,...,m_{k},\\
 \end{equation}
 \hspace*{5.0 pc} where  $m_k$ is the smallest  positive integer such that\\
\begin{equation}\label{eq:4.2a}
\| x^{k+1}_{m_k}-x^{k+1}_{m_k-1} \| \leq \delta\|x^{k-1}-x^k\|.
\end{equation}
\hspace*{1.5 pc} Step 2: {If $\|x^{k+1}_{m_k}-x^k\|/{\max\{1,\|x^k\|\}}\geq {\rm tol}$, set $$x^{k+1}=x^{k+1}_{m_k},$$
\hspace*{5.0 pc} otherwise, calculate $x^{k+1}_{m_k+1}$ via (\ref{eq:4.102a}). If $\|x^{k+1}_{m_k+1}-x^k\|/{\max\{1,\|x^k\|\}}<{\rm tol}$, stop; \\
\hspace*{5.0 pc}  {\rm otherwise, set}}}\\
 \hspace*{5.0 pc} {
$$x^{k+1}=x^{k+1}_{m_k+1}.$$}
\hspace*{5.0 pc} {\rm Set} $k:=k+1$ {\rm and return to Step 1.}
\end{algorithm}

\begin{remark}
\rm
There are two observations for the implement of cDCA.
\vskip 1mm

\noindent
(i) In principle, the initial point $x_0^{k+1}$ of  the {inner loop} (\ref{eq:4.102a}) can be arbitrarily selected in $\mathbb{R}^n$. However choosing a good initial point can undoubtedly reduce the numbers of the {inner loop}. Motivated by the multi-step inertial methods \cite{CL,Dong}, $x_0^{k+1}=x^{k}+\sum_{i\in \mathcal{S}}a_{i,k}(x^{k-i}-x^{k-i-1})$ is naturally a good choice, where  $\mathcal{S}\overset{\rm def}=\{0,1,\ldots,s\}$, $s\in \mathbb{N}$, and $\{a_{i,k}\}_{i\in \mathcal{S}}\in \mathbb{R}^{s+1}$ with $x^{-i}=x^0$, $i\in \mathcal{S}\backslash$$\{1\}$.
\vskip 1mm

\noindent
(ii) By Lemma 3.1,   for each $k\geq 0$, ${\rm Prox}_{\mu g}[T_k^{\lambda,\mu}]$ is  a  contraction with the coefficient $\frac{L_f}{2\lambda+L_f}$, thus we assert that (\ref{eq:4.102a}) terminates in a finite number of iterates provided $x^k\neq x^{k-1}$.
\end{remark}

\begin{remark}
\rm
Although similar inexact version of the DC problem was already discussed in  \cite{LeLe},  it is  very different from cDCA.  On the one hand,  in each iteration, the inexact version in \cite{LeLe} does not give a specific scheme for solving the convex  subproblem, while the core of cDCA is to try to give a good calculation scheme for the convex subproblem. On the other hand, in the $k$th iteration, the inexact version in \cite{LeLe} computes a $\varepsilon_k$-solution of the convex subproblem, where $\varepsilon_k>0$ is a given error such that $\sum_{k=1}^\infty \varepsilon_k<+\infty$ to  ensure the convergence of the algorithm, while cDCA establishes the iterative termination rule  in an adaptive way (\ref{eq:4.2a}) for the convex subproblem.
\end{remark}

The following lemma ascertains the
validity of the stopping criterion used in Step 2, that is, $x^k$ is a approximate critical point of $F$ when Algorithm 3.2 terminates.

\begin{lemma}
\label{lem43}
Let $ \{x^{k}\}_{k=0}^\infty$ be the sequence  generated by cDCA. If $x^{k+1}_{m_k}=x^k=x^{k+1}_{m_k+1}$ for some $k$, then $x^k$  is a  critical point of $F$.
\end{lemma}

\begin{proof}
Suppose $x^{k+1}_{m_k}=x^k=x^{k+1}_{m_k+1}$, then  ${\rm Prox}_{\mu g}[T_k^{\lambda,\mu}]x^{k}=x^{k}$, which means that $x^{k}$ is  the unique fixed point of ${\rm Prox}_{\mu g}[T_k^{\lambda,\mu}]$. In view of (\ref{eq:3.c}), it follows that ${0}\in \nabla f(x^{k})+\partial g(x^{k})-\partial h(x^{k})$, i.e., $x^{k}$ is a critical point of $F$.
\end{proof}

From Lemma \ref{lem43}, we see that if cDCA terminates in a finite (say $k$) step of
iterations, then $x^{k}$ is an approximate critical point of $F$. So in the convergence analysis of the next section, we
assume that cDCA does not terminate in any finite iterations, and hence generates
an infinite sequence $\{x^{k}\}_{k=0}^\infty$.

\section{Convergence analysis}

To show the convergence of  cDCA,
we introduce an auxiliary function
$$
E(x,y)=F(x)+\tau\|x-y\|^2,\quad\forall (x, y)\in \mathbb{R}^n\times \mathbb{R}^n,
$$
where $\tau=\frac{(1-\mu\lambda)^2\delta^2}{2\lambda \mu^2}$.  Obviously, we have
$$
E(x,x)=F(x),\quad\forall x\in \mathbb{R}^n.
$$
 From $\delta\in (0, \frac{2\lambda}{L_f})$ and $\mu=\frac{2}{2\lambda+L_f}$,  it follows
$\tau=\frac{L_f^2\delta^2}{8\lambda}<\frac{\lambda}{2}$.

{We firstly show the global subsequential convergence of the sequence $ \{x^{k}\}_{k=0}^\infty$  generated by cDCA.}

\begin{theorem}\label{th4.2}\
Let $ \{x^{k}\}_{k=0}^\infty$ be the sequence  generated by cDCA. Then the following statements hold:
\vskip 1mm

\noindent
{\rm (i)} The sequence $\{E(x^k,x^{k-1})\}_{k=0}^\infty$  is nonincreasing and satisfies
\begin{equation}
\label{L2-p}
E(x^{k+1},x^k)+(\frac{\lambda}{2}-\tau)\|x^k-{x}^{k+1}\|^2 \leq E(x^k,x^{k-1}),\,\,\forall k\geq 0;
\end{equation}
\vskip 1mm

\noindent
{\rm (ii)} $\omega (x^k)\subset \mathcal{X}$;
\vskip 1mm

\noindent
{\rm (iii)} $\lim_{k\rightarrow \infty}F(x^k)$ and $\lim_{k\rightarrow \infty}E(x^k, x^{k-1})$ exist, and
$
\lim_{k\rightarrow \infty}F(x^k)=\lim_{k\rightarrow \infty}E(x^{k},$ $x^{k-1})=F(\hat{x}),\,\,\,\,\forall \hat{x}\in \omega (x^k).
$
\end{theorem}

\begin{proof} (i) We consider separately two  possible values of $x^{k+1}$. For the case of $x^{k+1}=x^{k+1}_{m_k}$, using (\ref{eq:4.102a}) and (\ref{eq:3.b}), we have

\begin{equation}\label{new:1a}
\aligned
x^{k+1}_{m_k}&={\rm Prox}_{\mu g}[T_k^{\lambda,\mu}](x^{k+1}_{m_k-1})\\
& ={\rm Prox}_{\mu g}[(1-\mu\lambda)x^{k+1}_{m_k-1}-\mu\nabla  f(x^{k+1}_{m_k-1})+\mu\eta^k+\mu\lambda x^k],
\endaligned
\end{equation}
which implies
\begin{equation*}
 (1-\mu\lambda)x^{k+1}_{m_k-1}-\mu\nabla  f(x^{k+1}_{m_k-1})+\mu\eta^k+\mu\lambda x^k\in x^{k+1}_{m_k}+\mu\partial g(x^{k+1}_{m_k}).
\end{equation*}
Hence there exists  $\xi^{k+1}\in \partial g(x^{k+1}_{m_k})$ such that
\begin{equation}\label{eq:4.77a}
\aligned
\nabla f(x^{k+1}_{m_k})+\xi^{k+1}=&\frac{1-\mu\lambda}{\mu}(x^{k+1}_{m_k-1}-x^{k+1}_{m_k})+
\lambda(x^k-x^{k+1}_{m_k})\\
&+\nabla f(x^{k+1}_{m_k})-\nabla f(x^{k+1}_{m_k-1})+\eta^k.
\endaligned
\end{equation}
{From  the subdifferential inequality, it follows
\begin{equation}\label{eqde1}
\langle \nabla f(x^{k+1}_{m_k})+\xi^{k+1}, x^k-x^{k+1}_{m_k}\rangle\le f(x^k)+g(x^k)
-f(x^{k+1}_{m_k})-g(x^{k+1}_{m_k}),
\end{equation}
and
\begin{equation}\label{eqde2}
\langle\eta^k, x^{k}-x^{k+1}_{m_k}\rangle\geq h(x^k)-h(x^{k+1}_{m_k}).
\end{equation}
Combining \eqref{eq:4.77a}, \eqref{eqde1} and \eqref{eqde2}, we obtain}
\begin{equation}\label{eq:new2}
\aligned
&\lambda\|x^k-x^{k+1}_{m_k}\|^2\\
\leq &F(x^k)-F(x^{k+1}_{m_k}) -\langle \nabla
f(x^{k+1}_{m_k})-\nabla f(x^{k+1}_{m_k-1}), x^k-x^{k+1}_{m_k}\rangle\\
&+\frac{1-\mu\lambda}{\mu}\langle x^{k+1}_{m_k}-x^{k+1}_{m_k-1}, x^k-x^{k+1}_{m_k}\rangle\\
\leq &F(x^k)-F(x^{k+1}_{m_k})\\
&+\frac{1-\mu\lambda}{\mu}\langle (I-\frac{\mu}{1-\mu\lambda}\nabla f)(x^{k+1}_{m_k})-(I-\frac{\mu}{1-\mu\lambda}\nabla f)(x^{k+1}_{m_k-1}), x^k-x^{k+1}_{m_k}\rangle.
\endaligned
\end{equation}
Noting that $1-\frac{\mu}{1-\lambda\mu}=\frac2{L_f},$ by Lemmas \ref{lem22.2} and \ref{lem22.4},  $ I-\frac{\mu}{1-\lambda\mu}\nabla f   $ is nonexpansive.  We get from  (\ref{eq:4.2a}) that
\begin{equation}\label{eq:3.2ab}
\aligned
&\left|\frac{1-\mu\lambda}{\mu}\langle (I-\frac{\mu}{1-\mu\lambda}\nabla f)(x^{k+1}_{m_k})-(I-\frac{\mu}{1-\mu\lambda}\nabla f)(x^{k+1}_{m_k-1}), x^k-x^{k+1}_{m_k}\rangle\right|\\
\leq &\frac{\lambda}{2}\|x^k-x^{k+1}_{m_k}\|^2+ \frac{(1-\mu \lambda)^2\delta^2}
{2\lambda\mu^2}\|x^{k-1}-x^{k}
\|^2.
\endaligned
\end{equation}
Noting $x^{k+1}=x^{k+1}_{m_k}$ and $\tau=\frac{(1-\mu\lambda)^2\delta^2}{2\lambda \mu^2}$, it derives from (\ref{eq:new2}) and (\ref{eq:3.2ab})  that
\begin{equation}\label{eq3.6}
\frac{\lambda}{2}\|x^k-{x}^{k+1}\|^2 \leq F(x^k)-F(x^{k+1})+\tau\|x^{k-1}-x^k\|^2.
\end{equation}
For the case of $x^{k+1}=x^{k+1}_{m_k+1}$, by (\ref{eq:4.102a}), we get
\begin{equation*}
x^{k+1}_{m_k+1}={\rm Prox}_{\mu g}[T_k^{\lambda,\mu}](x^{k+1}_{m_k}).
\end{equation*}
Note that
\begin{equation*}
\| x^{k+1}_{m_k+1}-x^{k+1}_{m_k} \| \leq \delta\|x^{k-1}-x^k\|,
\end{equation*}
it is easy to see that (\ref{eq3.6}) can also be derived through a similar derivation as above. From (\ref{eq3.6}), we have \eqref{L2-p}  which implies that  the sequence $\{E(x^k,x^{k-1})\}_{k=0}^\infty$
is nonincreasing due to  $\tau<\frac{\lambda}{2}$.

(ii) From \eqref{eq3.6}, we have
\begin{equation}\label{eq3.77}
\aligned
(\frac{\lambda}{2}-\tau) \sum_{l=0}^k\|x^l-x^{l+1}\|^2\leq &F(x^0)-F(x^{k+1})+\tau \|x^{-1}-x^0\|^2\\
\leq &F(x^0)-F(x^*)+\tau \|x^{-1}-x^0\|^2,
\endaligned
\end{equation}
{where $x^*\in \mathcal{S}$.}
Since $\tau<\frac{\lambda}{2}$,  (\ref{eq3.77}) implies that
\begin{equation} \label{eq3.8}
\sum_{k=0}^\infty\|x^k-x^{k+1}\|^2<\infty \,\,\,{\rm and}\,\, \|x^k-x^{k+1}\|
\rightarrow 0\, \, (k\rightarrow \infty).
\end{equation}

From \eqref{L2-p}, $\{E(x^k, x^{k-1}):=F(x^k)+\tau\|x^{k-1}-x^k\|^2\}_{k=0}^\infty$
is strictly decreasing, which together with (\ref{eq3.8}) implies that
there exists a finite limit value $\lim_{k\rightarrow \infty}F(x^k)$. Since $F$
is level-bounded on $\mathbb{R}^n$, we assert that $\{x^k\}_{k=0}^\infty$ is bounded.
For any $\hat{x}\in \omega(x^k)$, there exists a subsequence
$\{x^{k_l}\}_{l=0}^\infty \subset \{x^k\}_{k=0}^\infty$ such that
$x^{k_l}\rightarrow \hat{x}$ as $l\rightarrow \infty.$
From (\ref{eq:4.77a}), we get
\begin{equation}\label{eq:3.506}
\aligned
\nabla f(x^{k_l+1}_{m_{k_l}})+\xi^{k_l+1}-\eta^{k_l}=
&\frac{1-\mu\lambda}{\mu}(x^{k_l+1}_{m_{k_l}-1}-x^{{k_l}+1}_{m_{k_l}})+
\lambda(x^{k_l}-x^{{k_l}+1}_{m_{k_l}})\\
&+\nabla f(x^{k_l+1}_{m_{k_l}})-\nabla f(x^{{k_l}+1}_{m_{k_l}-1}).
\endaligned
\end{equation}
By Lemma \ref{le2.22}, $\{\eta^{k_l}\}_{l=0}^\infty$ is  bounded {and therefore it has a
convergent subsequence.}
Without loss of generality, we assume that  $\{\eta^{k_l}\}_{l=0}^\infty$ is
convergent. Noting $x^{k_l+1}=x^{k_l+1}_{m_{k_l}}$, since
$x^{k_l}-x^{k_l+1}\rightarrow 0$,
it follows that ${x}^{k_l+1}\rightarrow \hat{x}$ and hence
$\nabla f({x}^{k_l+1})\rightarrow \nabla f(\hat{x})$ as $l\rightarrow \infty$.
 On the other hand, we assert from
(\ref{eq:4.2a}) and (\ref{eq3.8})  that $x^{k_l+1}_{m_{k_l-1}}-x^{k_l+1}\rightarrow 0$
and $\nabla f(x^{k_l+1})-\nabla f(x^{k_l+1}_{m_{k_l-1}})\rightarrow 0$ hold as
$l\rightarrow \infty$. Consequently, (\ref{eq:3.506})  together with these facts
concludes that  $\{{\xi}^{k_l+1}\}_{l=0}^\infty$ is also convergent. Letting $l\rightarrow \infty$ in (\ref{eq:3.506}),  we have
$$
\nabla f(\hat{x})+\hat{\xi}-\hat{\eta}={ 0},
$$
where $\hat{\eta}=\lim_{l\rightarrow \infty}\eta^{k_l}$ and $\hat{\xi}=\lim_{l\rightarrow \infty}{\xi}^{k_l+1}$. This means
that $\hat{x}\in \mathcal{X}$ since $\hat{\xi}\in \partial g(\hat{x})$ and
{$\hat{\eta}\in \partial h(\hat x)$} hold by using Lemma \ref{le2.3} and hence $\omega (x^k)\subset \mathcal{X}$
 follows.

{(iii) Taking $\hat{x}\in \omega (x^k)$ arbitrarily and  $\{x^{k_l}\}_{l=0}^\infty \subset \{x^k\}_{k=0}^\infty$ such that $x^{k_l}\rightarrow \hat{x}$ as $l\rightarrow \infty$,
from the lower semicontinuity of $F$, we get
\begin{equation}\label{new1}
F(\hat{x})\leq \varliminf_{l\rightarrow \infty}F(x^{k_l})=\lim_{k\rightarrow \infty} F(x^k)=\lim_{k\rightarrow \infty}E(x^{k},x^{k-1}).
\end{equation}
On the other hand, using subdifferential  inequality, we have
$$
g(\hat{x})\geq g(x^{k_l})+\langle \xi^{k_l}, \hat{x}-x^{k_l}\rangle,
$$
where $\xi^{k_l}\in \partial g(x^{k_l}).$ Consequently,
\begin{equation}\label{new2}
f(x^{k_l})+g(\hat{x})-h(x^{k_l})\geq F(x^{k_l})+\langle \xi^{k_l},
\hat{x}-x^{k_l}\rangle.
\end{equation}
{From the proof process in (ii), we can further assume that $\{\xi^{k_l}\}_{l=0}^\infty$ is convergent and its limit belongs to $\partial g(\hat{x})$. Noting this fact and the continuity of $f$ and $h$,} we obtain by making $l\rightarrow \infty$
in (\ref{new2}) that
$$
F(\hat{x})\geq \lim_{l\rightarrow \infty}F(x^{k_l})=
\lim_{k\rightarrow \infty} F(x^k)=\lim_{k\rightarrow \infty}E(x^{k},x^{k-1}).
$$
Hence this together with (\ref{new1}) leads to the desired result.
}
\end{proof}

{Now we consider the global convergence property of the whole sequence $\{x^{k}\}_{k=0}^\infty$.}
\vskip 1mm

{If $h$ is  differentiable at $\bar{x}\in \mathbb{R}^n$, it is easy
to verify that for any $y\in \mathbb{R}^n$
$$\partial_x E(\bar{x},y)=\nabla f(\bar{x})+\partial g(\bar{x})-\nabla h(\bar{x})+2\tau(\bar{x}-y),$$
$$\partial_y E(\bar{x},y)=2\tau  (y-\bar{x}),$$
and hence
$$
\partial E(\bar{x},y)=(\nabla f(\bar{x})+\partial g(\bar{x})-\nabla h(\bar{x})+2\tau (\bar{x}-y),\,2\tau (y-\bar{x})).
$$

\begin{lemma}
\label{lem44}
Let $\{x^{k}\}_{k=0}^\infty$ be the sequence  generated by cDCA.  {Assume that $\nabla h$ is  $L_h$-Lipschitz continuous  on  ${\rm co}(\{x^k\}_{k=0}^\infty)$.}
 Then there exist $w^{k+1}\in \partial F(x^{k+1})$ such that
$$
||w^{k+1}||\le \frac{(1-\mu\lambda)\delta}{\mu}||x^{k}-x^{k-1}||+(\lambda+L_{ h})||x^{k+1}-x^{k}||,
$$
and $v^{k+1}\in \partial E(x^{k+1},x^k)$ such that
$$
||v^{k+1}||\le \frac{(1-\mu\lambda)\delta}{\mu}||x^{k}-x^{k-1}||
+(\lambda+L_{ h}+4\tau )||x^{k+1}-x^{k}||.
$$
\end{lemma}

\begin{proof}
Let $\xi^{k+1}\in \partial g(x^{k+1}_{m_k})$ such that \eqref{eq:4.77a} {holds} and
$w^{k+1}=\nabla f(x^{k+1})+\xi^{k+1}-\nabla h(x^{k+1})$. Then we have from
 \eqref{eq:4.77a} that
\begin{equation}\label{add 1}
\aligned
w^{k+1}&=\frac{1-\mu\lambda}{\mu}\{(I-\frac{\mu}{1-\mu\lambda}\nabla f)(x^{k+1}_{m_k-1})-(I-\frac{\mu}{1-\mu\lambda}\nabla f)(x^{k+1})\}\\
&+\lambda (x^k-x^{k+1})+\nabla h(x^k)-\nabla h(x^{k+1}).
\endaligned
\end{equation}
Note that $ I-\frac{\mu}{1-\lambda\mu}\nabla f   $ is nonexpansive,  we get from  (\ref{add 1})  and  (\ref{eq:4.2a}) that
$$
||w^{k+1}||\le \frac{(1-\mu\lambda)\delta}{\mu}||x^{k}-x^{k-1}||+(\lambda+L_{ h})||x^{k+1}-x^{k}||.
$$
Let $v^{k+1}=(w^{k+1}+2\tau (x^{k+1}-x^k),2\tau (x^k-x^{k+1}))\in \partial E(x^{k+1},x^k)$. Then
$$
\aligned
\|v^{k+1}\|&\le\|w^{k+1}\|+4\tau \|x^{k+1}-x^k||\\
&\le  \frac{(1-\mu\lambda)\delta}{\mu}||x^{k}-x^{k-1}||+(\lambda+L_{ h}+4\tau )||x^{k+1}-x^{k}||.
\endaligned
$$
This completes the proof.
\end{proof}

\begin{remark}
\rm
There are some DC regularizers $g-h$ satisfying the assumption in Lemma \ref{lem44}   that arise in applications, such as MCP regularizer, SCAD regularizer and logarithmic regularizer (see  \cite{Wen} for more details).
\end{remark}

Using Theorem \ref{th4.2} (i) and (ii), Lemma \ref{lem44} and \cite[Theorem 3.7]{OCTP}, we get the convergence of $\{x^{k}\}_{k=0}^\infty$ to a critical point of $F$.
}

\begin{theorem}
\label{th45}
Let $\{x^{k}\}_{k=0}^\infty$ be the sequence  generated by cDCA and {$\hat x$ be some point in $\omega(x^k)$.
 Suppose that  $E$ has the
Kurdyka--Lojasiewicz property at $(\hat{x}, \hat{x})$ and
$\nabla h$ is  Lipschitz continuous on
${\rm co}(\{x^k\}_{k=0}^\infty)$.}
Then $\{x^{k}\}_{k=0}^\infty$ converges to a critical point of $F$.
\end{theorem}

\begin{remark}
\rm
If we further assume that the function $\phi$ of $E$ in K\L ~property has the form $\phi(s)=cs^{1-\theta}, \theta\in [0,1)$, we can get the convergence rate of Algorithm 3.2 by a similar analysis with \cite{Wen,{OCTP}}.
\end{remark}

\section{ Numerical Experiments}

In order to show the  practicability and effectiveness of cDCA, we present some numerical experiments via implementing cDCA, ADCA and pDCA$_e$ for solving the problem \eqref{DC}. In cDCA, we take $x_0^{k+1}=x^{k}+\alpha_k(x^{k}-x^{k-1})+\beta_k(x^{k-1}-x^{k-2})$, $\forall k\in \mathbb{N}$.
\vskip 2mm

In our numerical tests, we focus on the following DC regularized least squares problem:
\begin{equation}
\label{lsp}
\min_{x\in \mathbb{R}^{n}}F(x)=\frac12\|Ax-b\|^2+g(x)-h(x),
\end{equation}
where $A\in \mathbb{R}^{m\times n}, b \in \mathbb{R}^m$, $g$ is a proper closed convex function and $h$ is a continuous convex function. Observe that $L_f = \lambda_{\max}(A^TA)$.

We consider two different classes of regularizers: the $\ell_{1-2}$ regularizer  and the logarithmic regularizer.
The ${\ell}_{1-2}$ regularized least squares problem is defined as:
\begin{equation}
\label{l12}
\min_{x\in \mathbb{R}^{n}}F_{\ell_{1-2}}(x)=\frac12\|Ax-b\|^2+\gamma\|x\|_1-\gamma\|x\|,
\end{equation}
where $\gamma > 0$ is the regularization parameter. This problem takes the form of \eqref{DC} with $g(x)=\gamma\|x\|_1$ and $h(x)=\gamma\|x\|$. We suppose in addition
that $A$ in (\ref{l12}) does not have zero columns. Using this
hypothesis, we see that $F_{\ell_{1-2}}$ is level-bounded, and that if we choose
$\gamma < \frac12\|A^Tb\|_\infty$, then assumptions in Theorem \ref{th45} are satisfied (see Subsection 5.1 in  \cite{Wen}).  We set $\gamma=0.01$ in \eqref{l12}.
The least squares problem with the logarithmic regularization function is defined as:
\begin{equation}\label{log}
\min_{x\in \mathbb{R}^{n}}F_{\log}(x)=\frac12\|Ax-b\|^2
+\sum_{i=1}^t\left[\gamma\log{(}|x_i|+\epsilon)-\gamma\log\epsilon\right],
\end{equation}
where  $\epsilon>0$ is a constant, and $\gamma > 0$ is the regularization parameter. Take $g(x)=\frac{\gamma}{\epsilon}\|x\|_1$ and $h(x)=\sum_{i=1}^t\gamma\left[\frac{|x_i|}{\epsilon}-\log(|x_i|+\epsilon)+\log\epsilon\right]$.
It is not hard to show that $F_{\log}$ is level-bounded and assumptions in Theorem \ref{th45} are satisfied (see Subsection 5.2 in  \cite{Wen}). We set $\gamma=0.01$ and $ \epsilon= 0.5$ in \eqref{log}.

We present numerical experiments for solving \eqref{lsp} on random instances generated as follows. We first generate
an $m\times n$ matrix $A$ with i.i.d. standard Gaussian entries, and then normalize this matrix
so that the columns of $A$ have unit norms. A subset $S$ of the size $K$ is then chosen uniformly
at random from $\{1, 2, 3, \ldots, n\}$ and a $K$-sparse vector $x^*$ having i.i.d. standard Gaussian
entries on $S$ is generated. Finally, we set $b = Ay +0.001\cdot{\hat n}$, where $\hat n \in \mathbb{R}^m$ is a random vector with i.i.d. standard Gaussian entries.

In the numerical
results listed in the following tables, `Iter' and `InIt' denote  the number of iterations of the outer loop and
the  number  of the additional iterations needed {when \eqref{eq:4.2a} is not satisfied}, i.e., InIt=$\sum_{k=1}^{\rm Iter}(m_k$ $-1)$, respectively and the sum of Iter and InIt is denoted by `tIter'.  %We terminate cDCA when 
%\begin{equation}
%\label{terminate-Cond}
%\frac{\min\{\|x^k_{m_{k-1}}-x^{k-1}\|,\|x^k_{m_{k-1}+1}-x^{k-1}\|\}}{\max\{1,\|x^k\|\}}<\varepsilon,
%\end{equation}
%and other methods when 

\subsection{Optimal choice of the parameter $\lambda$}

The parameter $\lambda$ is crucial to the computational speed of cDCA. Firstly, the outer loop is slow   when $\lambda$ is big, because  the stepsize $\mu$ becomes small  when $\lambda$ increases.  Secondly, the contraction coefficient of $T^{\lambda,\mu}_k$ in \eqref{eq:4.102a} is $\frac{L_f}{2\lambda+L_f}$, which depends on $\lambda$ and is small when $\lambda$ is big. So, the number of the inner loop is small when $\lambda$ is big. Therefore, the parameter $\lambda$ can't be very big or very small to achieve the fast convergence.

In the first example, we test cDCA with the parameter $\lambda$ varying among $\{ 0.001L_f,$ $0.01L_f,0.1L_f,0.5L_f\}$  for the $\ell_{1-2}$ regularizer and $\{0.07L_f,0.1L_f,0.2L_f\}$   for the logarithmic regularizer.
We take
${\rm tol}=10^{-5}$ for the $\ell_{1-2}$ regularizer with  $(m,n,K)=(120i,512i,20i),$ $i=1,4,7,10$. For the logarithmic regularizer, we take
${\rm tol}=10^{-4},5\times10^{-5},2\times10^{-5}$  for $(m,n,K)=(120i,512i,20i),$ $i=1$, $i=4$ and  $i=7,10$, respectively. The
starting point $x^0$ is randomly generated in $(0,1)^n$ and $x^{-1}=x^0+{\rm ones}(n,1)$, $x^{-2}=x^0$. Set $\delta=\frac{1.99\lambda}{L_f}$, $\mu=\frac{2}{2\lambda+L_f}$,  and $\alpha_k=\beta_k=0.6$, for all $ k\ge0.$
We report  Iter and InIt, averaged
over  30 random instances. In the tables, ``$\Max$" means that tIter hits 100000.

\begin{table}[!h]
\caption{Comparisons of cDCA with different $\lambda$ for solving (\ref{l12})}\label{table1}
{\scriptsize
\begin{center}
\begin{tabular}{c c c c c c c c c c c c c c c c c c c c c c c c c c}
\hline
$(m,n,K)$&&\multicolumn{2}{c}{(120,512,20)}&\multicolumn{4}{c}{(480,2048,80)}
&\multicolumn{2}{c}{(840,3584,140)}&\multicolumn{3}{c}{(1200,5120,200)}\\
\cline{1-1} \cline{3-4} \cline{6-7} \cline{9-10} \cline{12-13}
$\lambda$&&Iter&InIt&& Iter&InIt&&Iter&InIt&&Iter&InIt\\
\hline
\hline
%0.0001$L_f$&&2&1608&&3&6946&&1&39708&&1&47293\\
0.001$L_f$&&169.2&97039.1&&554.3&89428.7&&$\Max$&$\Max$&&$\Max$&$\Max$\\
0.01$L_f$&&709.2& 8147&& 1252.9& 46104.7&& 1764.5&97732.3&&$\Max$&$\Max$\\
0.1$L_f$&&1916.6& 293.6&&4325.5& 346.7&&5955.4&383.9&& 6657.4& 360.7\\
0.5$L_f$&&88871.6&6.1&& $\Max$&$\Max$&& $\Max$&$\Max$&&$\Max$&$\Max$\\
\hline
\end{tabular}
\label{Tab1}
\end{center}
}
\end{table}

\begin{table}[!h]
\caption{Comparisons of cDCA with different $\lambda$ for solving \eqref{log}}\label{table2}
{\scriptsize
\begin{center}
\begin{tabular}{c c c c c c c c c c c c c c c c c c c c c c c c c c}
\hline
$(m,n,K)$&&\multicolumn{2}{c}{(120,512,20)}&\multicolumn{4}{c}{(480,2048,80)}
&\multicolumn{2}{c}{(840,3584,140)}&\multicolumn{3}{c}{(1200,5120,200)}\\
\cline{1-1} \cline{3-4} \cline{6-7} \cline{9-10} \cline{12-13}
$\lambda$&&Iter&InIt&& Iter&InIt&&Iter&InIt&&Iter&InIt\\
\hline
\hline
0.07$L_f$  &&1545.2&1283.8&&3173.2& 1819.9&&4997.4&2760.6&&6870.8&3711.4\\
0.1$L_f$   && 1811.4&781.1&&2726.4& 274.4&&3608.8&357.6&&4325.2&309.7\\
0.2$L_f$      &&12145.3&38.1&&17676.3&40.5&&$\Max$&$\Max$&&$\Max$&$\Max$\\
\hline
\end{tabular}
\label{Tab2}
\end{center}
}
\end{table}

As seen in Tables \ref{Tab1} and \ref{Tab2}, the optimal choice of the parameter $\lambda$ is $0.1L_f$ for the $\ell_{1-2}$ regularizer and   the logarithmic regularizer, for which tIter is the smallest. It is also observed  that InIt decreases and Iter increases when  $\lambda$ becomes big, which is consistent with the foregoing analysis.

\subsection{Comparison with ADCA and pDCA$_e$ }

In the second experiment, we compare cDCA with ADCA and  pDCA$_e$  studied in various work such as \cite{Wen,{PhanLe}}  for solving the  regularized least squares problem \eqref{lsp} with the $\ell_{1-2}$ regularizer and the logarithmic regularizer.  cDCA,  ADCA and pDCA$_e$ are all proximal-gradient type algorithms, which involve proximal operators and gradient operators. We terminate ADCA and pDCA$_e$  when
\begin{equation}
\label{terminate-Cond2}
\frac{\|x^k-x^{k-1}\|}{\max\{1,\|x^{k-1}\|\}}<{\rm tol}
\end{equation}
where ${\rm tol}$ is a small enough constant and will be given in the implement. 
 We discuss the implementation details of  three algorithms
below.

\textit{cDCA}\,\, For this algorithm, we take $L_f = \lambda_{\max}(A^TA)$, $\lambda=0.1L_f$, $\delta=\frac{1.99\lambda}{L_f}$, $\mu=\frac{2}{2\lambda+L_f}$ and $\alpha_k=\beta_k=0.6$, for all $ k\ge0.$ The
starting point $x^0$ is  randomly generated in $(0,1)^n$ and $x^{-1}=x^0+{\rm ones}(n,1)$, $x^{-2}=x^0$. We take 
${\rm tol} = 10^{-6}$ for the $\ell_{1-2}$ regularizer, and ${\rm tol}  = 6.5\times10^{-5}$  for $(m,n,K)=(120i,512i,20i),$ $i=1,2,3$  and ${\rm tol} = 2\times10^{-5}$  for other problem sizes of the logarithmic regularizer.}

\textit{ADCA}\,\, This algorithm was proposed in \cite{PhanLe}, and its scheme is given by \eqref{ADCA}.
 We take $q=5$, $L_f = \lambda_{\max}(A^TA)$, $\rho=1.1L_f$, and the
starting point $x^0$ is  randomly generated in $(0,1)^n$. We terminate it when \eqref{terminate-Cond2} holds  with $\varepsilon = 10^{-6}$ for the $\ell_{1-2}$ regularizer, \eqref{terminate-Cond2} holds with $\varepsilon = 6.5\times10^{-5}$  for $(m,n,K)=(120i,512i,20i),$ $i=1,2,3$ and $\varepsilon = 2\times10^{-5}$ for other problem sizes of the logarithmic regularizer.

\textit{pDCA$_e$}\,\, This algorithm was proposed in \cite{Wen}, and its scheme is given by \eqref{pDCAe}.
 We take $L_f = \lambda_{\max}(A^TA)$, choose the extrapolation
parameters $\{\beta_k\}_{k=0}^\infty$ and perform  the
adaptive restart strategies as in \cite{Wen}.  The
starting point $x^0$ is  randomly generated in $(0,1)^n$ and $x^{-1}=x^0$.  We terminate it when \eqref{terminate-Cond2} holds with $\varepsilon = 10^{-6}$  for the $\ell_{1-2}$ regularizer, \eqref{terminate-Cond2} holds with $\varepsilon = 1.5\times10^{-5}$ for $(m,n,K)=(120i,512i,20i),$ $i=1,2,3$ and $\varepsilon = 4\times10^{-6}$  for other problem sizes of the logarithmic regularizer.

In our numerical experiments below, we consider $(m,n,K) = (120i, 512i,20i)$
for $i = 1, 2,\ldots,10$. For each triple $(m,n, K)$, we generate 30 instances randomly as described above.   We report  tIter, Iter, CPU
times in seconds (CPU time), and the function values at termination (fval), averaged
over  30 random instances. Note that pDCA$_e$ and ADCA don't include an inner loop. Therefore, we compare  tIter  of cDCA and Iter of pDCA$_e$ and ADCA.
\begin{table}[h]
\label{table}
%\scriptsize
\footnotesize
\renewcommand\arraystretch{1.5}
\setlength\tabcolsep{2.2pt}
\centering
\caption{Computational results for solving (\ref{l12}) with cDCA, ADCA and pDCA$_e$}
\vskip 2mm
\begin{tabular}{cccccccccccccccccccc}
\hline
\multicolumn{3}{c}{Problem size}  & \multicolumn{2}{c}{cDCA} & \multicolumn{1}{c}{ADCA} &  \multicolumn{1}{c}{pDCA$_e$}&
\multicolumn{4}{c}{CPU time}  &\multicolumn{4}{c}{fval}\\
 \cline{1-3} \cline{5-7} \cline{9-11}\cline{13-15}
$m$ & $n$ & $K$ && tIter  & Iter &  Iter &&  cDCA&ADCA& pDCA$_e$  &&  cDCA&ADCA&  pDCA$_e$    \\
\hline
 \hline
120&512&20&&  2832.2&2818.2&5187.5 &&  1.846&2.304&3.675 && 0.152&0.152&0.152\\
240&1024&40&& 4006.8&4026.1&7373.9 && 4.434&5.694&8.384  &&  0.328&0.328&0.328\\
360&1536&60&& 4939&4955&9028.1 && 8.298&11.37&15.42  &&  0.510&0.510&0.510\\
480&2048&80&& 5332.1&5647.7&10418 && 12.87&19.41&25.62  &&  0.683&0.683&0.683\\
600&2560&100&& 6215.6&6374.8&11759.1 && 33.83&55.61&64.44  &&  0.893&0.893&0.893\\
720&3072&120&& 6771.2&7238.7&13039.1 && 99.37&192.0&194.9  &&  1.066&1.066&1.066\\
840&3584&140&& 7157.6&7757.8&14264.6 && 156.6&323.9&313.9  && 1.276&1.276&1.276\\
960&4096&160&& 7693.7&8559.8&15374.1 && 218.7&483.2&438.6 && 1.461&1.462&1.461\\
1080&4608&180&& 8151.6&9206.3&16460.9 && 299.7&687.9&607.5  && 1.651&1.651&1.651\\
1200&5120&200&& 8483.3&9862.7&17535.4 && 380.5&912.1&787.4 && 1.820&1.821&1.820\\
 \hline
\hline
\end{tabular}\label{Tab3}
\end{table}

\begin{table}[h]
\label{table}
%\scriptsize
\footnotesize
\renewcommand\arraystretch{1.5}
\setlength\tabcolsep{2.2pt}
\centering
\caption{Computational results for solving (\ref{log}) with cDCA, ADCA and pDCA$_e$}
\vskip 2mm
\begin{tabular}{cccccccccccccccccccc}
\hline
\multicolumn{3}{c}{Problem size}  & \multicolumn{2}{c}{cDCA} & \multicolumn{1}{c}{ADCA} &  \multicolumn{1}{c}{pDCA$_e$}&
\multicolumn{4}{c}{CPU time}  &\multicolumn{4}{c}{fval}\\
 \cline{1-3} \cline{5-7} \cline{9-11}\cline{13-15}
$m$ & $n$ & $K$ && tIter  & Iter &  Iter &&  cDCA&ADCA& pDCA$_e$  &&  cDCA&ADCA&  pDCA$_e$    \\
\hline
 \hline
120&512&20 && 2037.3&2044.3&3086.2 && 2.145&2.714&3.780 && 0.155&0.155&0.155\\
240&1024&40 && 2188.1&2929.5&4710 && 3.826&6.232&8.755 && 0.337&0.337&0.337\\
360&1536&60 && 2556.3&3370.7&5975.5 && 6.747&11.23&16.66 && 0.516&0.528&0.511\\
480&2048&80 && 3436.0&4288.7&7323.5 && 11.96&20.46&27.17 && 0.696&0.696&0.696\\
600&2560&100 && 3527.6&4760.4&8389.9 && 23.99&49.58&59.29 && 0.890&0.890&0.890\\
720&3072&120 && 3760.7&5210.3&9336.7 && 53.25&138.0&134.5 && 1.079&1.079&1.079\\
840&3584&140 && 3973.5&5609.3&10307.1 && 83.98&228.9&220.3 && 1.280&1.280&1.280\\
960&4096&160 && 4207.0&6048.1&11213.5 && 118.2&338.0&318.9 && 1.461&1.461&1.460\\
1080&4608&180 && 4407.1&6334.7&12081.1 && 160.5&465.9&443.4 && 1.656&1.656&1.655\\
1200&5120&200 && 4601.0&6873.5&12859.1 && 199.7&604.3&561.4 && 1.827&1.831&1.826\\
 \hline
\hline
\end{tabular}\label{Tab4}
\end{table}

We can see from Tables \ref{Tab3} and \ref{Tab4} that cDCA always outperforms ADCA and pDCA$_e$ in terms of iterations and CPU time for the $\ell_{1-2}$ regularizer and the logarithmic regularizer. The function values  at termination  of cDCA are almost the same as ADCA and pDCA$_e$ for two regularizers.

To illustrate the ability of recovering the original sparse solution by cDCA,
we plot in Figure \ref{fig1}  the true solution and the solutions obtained on
a random instance $(m, n, K)$ = (600,2560,100). The true solution  is represented
by asterisks, while circles are the estimates obtained by  cDCA. We see that the
estimates obtained by cDCA are quite close to the true values. For the same problem size as Figure 1, Figure \ref{fig2} shows the decay of the objective values with CPU time of the three algorithms.

\begin{figure}[!h]
\begin{center}
\includegraphics[width=7.4cm,height=5.5cm]{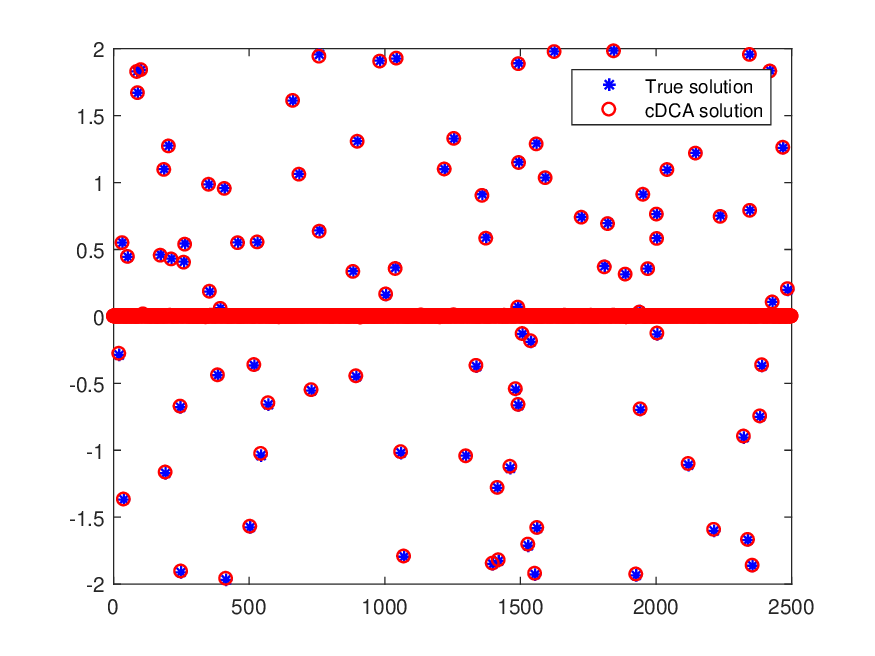}
\includegraphics[width=7.4cm,height=5.5cm]{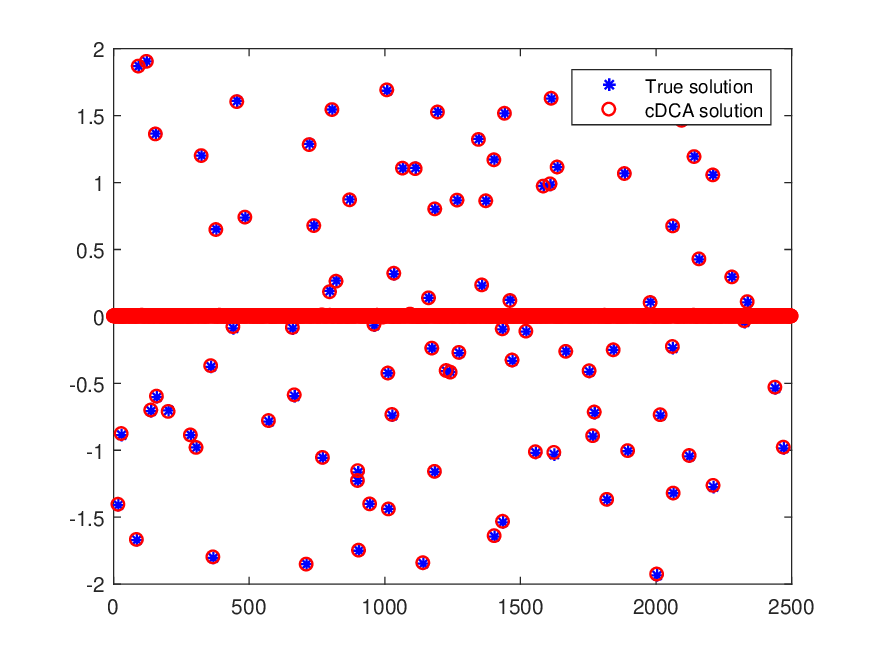}
\caption{The true solution and the solution by  cDCA  for $\ell_{1-2}$ regularizer (left) and logarithmic regularizer (right).  }\label{fig1}
\end{center}
\end{figure}

\begin{figure}[!h]
\begin{center}
\includegraphics[width=7.4cm,height=5.5cm]{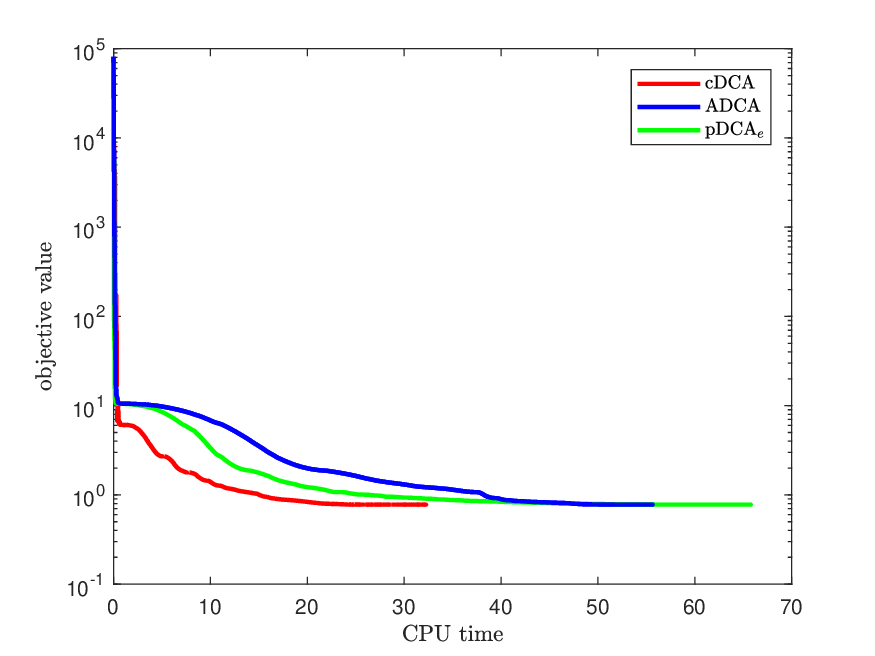}
\includegraphics[width=7.4cm,height=5.5cm]{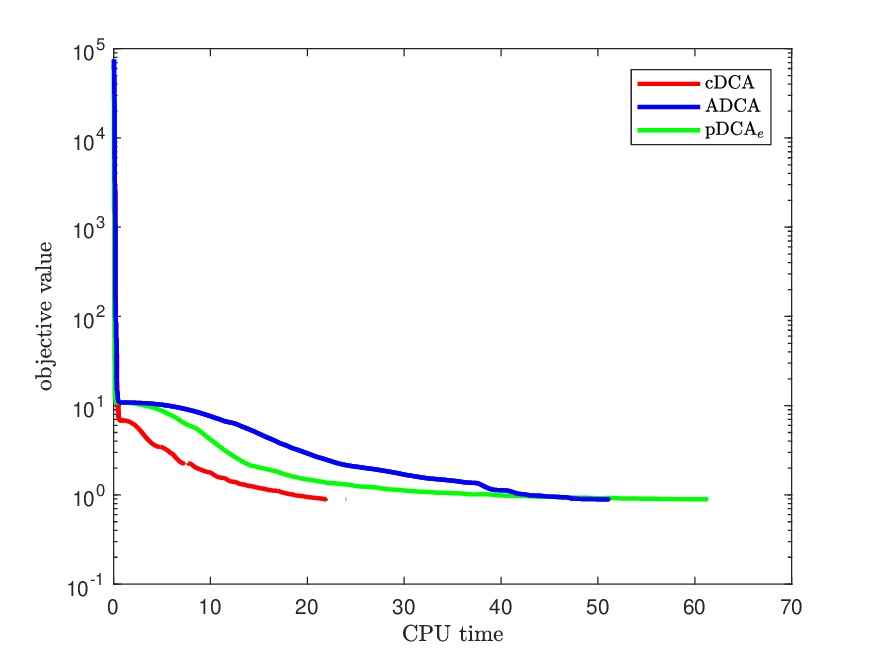}
\caption{The decay of objective values of cDCA, ADCA and pDCA$_e$ with CPU time  for $\ell_{1-2}$ regularizer (left) and logarithmic regularizer (right).  }\label{fig2}
\end{center}
\end{figure}

{Note that we don't compare cDCA with pDCA$_{e}$ since pDCA$_{e}$ is an extrapolation modification of pDCA while there is no extrapolation in cDCA.}

\section{Concluding Remarks}

In this paper, to better solve the subproblem of  DCA, we introduce a contractive difference-of-convex algorithms (cDCA) where an adaptive termination rule is given. The global subsequential convergence and the convergence of the whole iterative sequence generated by cDCA are established under appropriate conditions.

One interesting topic is how to further accelerate the cDCA through other ways, such as incorporating the extrapolation techniques with the inner loop, which deserves further research. It is also worth investigating to improve cDCA to solve more general DC problem.

\section*{Acknowledgements}

We were deeply grateful to the two anonymous referees for their constructive
suggestions and critical comments on the manuscript, which helped us signiﬁcantly
improve the presentation of this paper. We are grateful to Miss Ziyue Hu for valuable
help in numerical experiments.

This work was supported by
 Open Fund of Tianjin Key Lab for Advanced Signal Processing (No. 2022ASP-TJ01).
Dong was supported in part by National Natural Science Foundation of China (No. 12271273).


\vskip 2mm


\begin{thebibliography}{00}

\bibitem{Alvarado}
Alvarado, A., Scutari, G., Pang, J.S.: A new decomposition method for multiuser DC-programming
and its applications. \textit{IEEE Trans. Signal Process.} \textbf{62}, 2984--2998 (2014)


\bibitem{AT}
An, L.T.H., Tao, P.D.: The DC (difference of convex functions) programming and DCA revisited with DC models of real world nonconvex optimization problems. \textit{Ann. Oper. Res.} \textbf{133}, 23--46 (2005)


\bibitem{AV}
Arag\'on Artacho, F.J., Vuong, P.T.: The boosted difference of convex functions algorithm
for nonsmooth functions, \textit{SIAM J. Optim.} \textbf{30}, 980--1006 (2020)

\bibitem{Banert-Bot}
Banert, S., Bot, R.I.:
A general double-proximal gradient algorithm for d.c.
programming. \textit{Math. Program.}  \textbf{178}, 301--326 (2019)

%\bibitem{Bertsekas}
%Bertsekas, D.P.: \textit{Convex Analysis and Optimization}, Athena Scientific, Belmont,
%MA,  With A. Nedi\'{c} and A.E. Ozdaglar, 2003.

\bibitem{BC2011} Bauschke, H.H.,  Combettes, P.L.: \textit{Convex Analysis and Monotone Operator Theory in Hilbert Spaces}, 2nd ed. Springer, New York, 2017.

\bibitem{BT}
Beck, A., Teboulle, M.: A fast iterative shrinkage-thresholding algorithm for linear inverse problems.
\textit{SIAM J. Imaging Sci.} \textbf{2}, 183--202 (2009)

\bibitem{Gotoh}
Gotoh, J., Takeda, A., Tono, K.: DC formulations and algorithms for sparse optimization problems.
\textit{Math. Program.} \textbf{169}, 141--176 (2018)

\bibitem{dO}
de Oliveira, W.: The ABC of DC Programming, \textit{Set-Valued Var. Anal.} \textbf{28}, 679--706 (2020).

\bibitem{Dong}
Dong, Q.L., Huang, J.Z., Li, X.H. et al.: MiKM: multi-step inertial Krasnosel'ski\v{\i}--Mann algorithm and its applications. \textit{J. Glob. Optim.} \textbf{73}, 801--824 (2019)



\bibitem{LeLe}
Le Thi, H.A., Le, H.M., Phan, D.N., Tran, B.:
Stochastic DCA for minimizing a large sum of DC functions with application to multi-class logistic regression, \textit{Neural Networks}, \textbf{132},  220--231 (2020)


\bibitem{LeThi}
 Le Thi, H.A., Nguyen, V.V., Ouchani, S.:
Gene selection for cancer classification using DCA.  \textit{4th International Conference on Advanced Data Mining and Applications.} \textbf{5139},  62--72 (2008)

\bibitem{Le-Pham}
Le Thi, H.A., Pham Dinh, T.: Open issues and recent advances in DC programming and DCA. \textit{J. Glob. Optim.} (2023). https://doi.org/10.1007/s10898-023-01272-1.

\bibitem{Pham-Le2}
 Le Thi, H.A., Pham, D.T.: DC programming and DCA: thirty years of
developments. \textit{Math. Program.} \textbf{169}, 5--68  (2018)

\bibitem{LL}
Lin, D., Liu, C.: The modified second APG method for DC optimization problems. \textit{Optim. Lett.} \textbf{13}, 805--824 (2019)

\bibitem{Liu-Pong-Takeda}
Liu, T., Pong, T.K., Takeda, A.:
A refined convergence analysis of pDCAe with applications
to simultaneous sparse recovery and outlier detection. \textit{Comput. Optim. Appl.}  \textbf{73}, 69--100 (2019)

\bibitem{Liu-Pong-Takeda1}
Liu, T., Pong, T.K., Takeda, A.: A successive difference-of-convex approximation method for a class
of nonconvex nonsmooth optimization problems. \textit{Math. Program.}  \textbf{176}, 339--367   (2019)

\bibitem{Lou-Yan}
Lou, Y., Yan, M.: Fast L1-L2 Minimization via a proximal operator. \textit{J. Sci. Comput.} \textbf{74}, 767--785 (2018)

\bibitem{Lu-Zhou}
Lu, Z., Zhou, Z.: Nonmonotone enhanced proximal DC algorithms for a class of structured nonsmooth
DC programming. \textit{SIAM J. Optim.} \textbf{29(4)}, 2725--2752 (2019)

\bibitem{Lu-Zhou-Sun}
Lu, Z., Zhou, Z., Sun, Z.: Enhanced proximal DC algorithms with extrapolation for a class of structured nonsmooth DC minimization. \textit{Math. Program.}   \textbf{176}, 369--401 (2019)

\bibitem{Nesterov}
Nesterov, Y.E.: A method for solving the convex programming problem with convergence rate $O(1/k^2)$,
\textit{Dokl. Akad. Nauk SSSR}, \textbf{269}, 543--547 (1983) (in Russian).

\bibitem{Oblomskaja}
Oblomskaja, L.: Methods of successive approximation for linear equations in Banach spaces,
\textit{USSR Compt. Math. and Math. Phys.} \textbf{8}, 239--253 (1968)

\bibitem{OCTP}
Ochs, P.,  Chen, Y.,  Brox, T.,  Pock, T.: iPiano: Inertial proximal algorithm for nonconvex optimization, \textit{SIAM J. Imaging Sci.} \textbf{7(2)}, 1388--1419 (2014)

\bibitem{Pang-Razaviyayn-Alvarado}
Pang, J.S., Razaviyayn, M., Alvarado, A.: Computing B-stationary points of nonsmooth DC programs.
\textit{Math. Oper. Res.} \textbf{42(1)}, 95–118 (2017)


\bibitem{Pham}
Pham, T.N., Dao, M.N., Amjady, N., Shah, R.: A proximal splitting algorithm for generalized DC programming with applications in signal recovery, 2024 https://arxiv.org/abs/2409.01535


\bibitem{Pham-Le}
Pham, D.T., Le Thi, H.A.: Convex analysis approach to DC programming: theory, algorithms and
applications. \textit{Acta Math. Vietnam.} \textbf{22}, 289--355 (1997)

\bibitem{Pham-Le1}
Pham, D.T., Le Thi, H.A.: A D.C. optimization algorithm for solving the trust-region subproblem.
\textit{SIAM J. Optim.} \textbf{8}, 476--505 (1998)

\bibitem{Pham-Souad}
Pham Dinh, T., Souad, E.B.: \textit{Algorithms for solving a class of nonconvex optimizations problems:
methods of subgradient}. Fermat Day 85: Mathematics for Optimization, North Holland 1986


\bibitem{PhanLe}
Phan, D.N., Le, H.M., Le Thi, H.A.: Accelerated difference of convex functions algorithm and its application
to sparse binary logistic regression. In: Proceedings of the 27th International Joint Conference on
Artificial Intelligence, IJCAI-18, pp. 1369–1375. International Joint Conferences on Artificial Intelligence Organization (2018)


\bibitem{CL}
Poon, C..  Liang, J.: Geometry of first-order methods and adaptive acceleration, arXiv:2003 .03910.


\bibitem{Rock}
Rockafellar, R.T.: \textit{Convex Analysis}, Princeton University Press, Princeton, 1970

\bibitem{Sun}
Sun, W.Y., Sampaio, R.J.B., Candido, M.A.B.: Proximal point algorithm for
minimization of DC Functions. \textit{J. Comput. Math.}
\textbf{21(4)}, 451--462 (2003)


\bibitem{Syrtseva}
Syrtseva, K.,  de Oliveira, W.,  Demassey, S.,  van Ackooij, W.:  Minimizing the difference of convex and weakly convex functions via bundle method. \textit{Pac. J. Optim.} \textbf{20(4)}, 499-741, (2024)


\bibitem{Tuy}
Tuy, H.: \textit{Convex Analysis and Global Optimization,} 2nd edn. Springer, Berlin 2016

\bibitem{Wen}
Wen, B., Chen, X., Pong, T.K.: A proximal difference-of-conex algorithm with extrapolation. \textit{Comput.
Optim. Appl.} \textbf{69}, 297–324 (2018)

\bibitem{XuHK}
Xu, H.K.: Averaged mappings and the gradient-projection algorithm. \textit{J. Optim. Theory Appl.} \textbf{150}, 360--378 (2011)

\bibitem{Ye}
Ye, J.J., Yuan, X., Zeng, S., Zhang, J.: Difference of convex algorithms for bilevel programs with applications in hyperparameter selection. \textit{Math. Program.} \textbf{198}, 1583--1616 (2023)

\bibitem{Yin-Lou-He-Xin}
Yin, P., Lou, Y., He, Q., Xin, J.: Minimization of $\ell_{1-2}$ for compressed sensing. \textit{SIAM J. Sci. Comput.} \textbf{37(1)}, A536--A563 (2015)


\bibitem{Yu-Pong-Lu2021}
Yu, P., Pong, T.K., Lu, Z.: Convergence rate analysis of a sequential convex programming method with line search for a class of constrained difference-of-convex optimization problems, \textit{SIAM J. Optim.} \textbf{31}, 2024--2054 (2021)



\end{thebibliography}
\end{document}